\documentclass{pspum-l}

\usepackage[leqno]{amsmath}
\usepackage{amssymb}
\usepackage{amsthm}
\usepackage{mathrsfs}
\usepackage[shortalphabetic]{amsrefs}
\usepackage[all]{xy}

\newcommand{\R}{\mathbb{R}}
\newcommand{\spec}{\mathrm{spec}\,}
\newcommand{\supp}{\mathrm{supp}\,}
\newcommand{\re}{\,\mathrm{Re}\,}
\newcommand{\im}{\,\mathrm{Im}\,}

\newtheorem{thm}{Theorem}[section]
\newtheorem{lem}[thm]{Lemma}
\newtheorem{kcl}[thm]{Key Convexity Lemma}

\newtheorem{cor}[thm]{Corollary}

\theoremstyle{remark}
\newtheorem{rem}[thm]{Remark}

\theoremstyle{definition}

\numberwithin{equation}{section}

\begin{document}

\title[Quantitative Unique Continuation \ldots]{Quantitative Unique Continuation, Logarithmic Convexity of Gaussian Means and Hardy's Uncertainty Principle}

\author{Carlos E. Kenig}
\address{Depatment of Mathematics, University of Chicago, Chicago, IL 60637,USA}
\email{cek@math.uchicago.edu}
\thanks{Partially supported by NSF Grant \#DMS-0456583.}

\maketitle

In this paper we describe some recent works on quantitative unique continuation for elliptic, parabolic and dispersive equations. We also discuss recent works on the logarithmic convexity of Gaussian means of solutions to Schr\"odinger evolutions and the connection with a 
well-known version of the uncertainty principle, due to Hardy. The elliptic results are joint work with J. Bourgain \cite{BK}, while the remainder of the works discussed here are joint works with L. Escauriaza, G. Ponce and L. Vega (\cite{EKPV1}, \cite{EKPV2}, \cite{EKPV3},
\cite{EKPV4}, \cite{EKPV5}). The paper is based on lectures presented at WHAPDE 2008, Merida, Mexico. I am grateful to the organizers of WHAPDE 2008 and to the participants in the workshop for the invitation and the very friendly atmosphere of the workshop. For further references and background on the problems discusses here, see \cite{BK}, \cite{K1}, \cite{K2}, \cite{EKPV1}, \cite{EKPV2}, \cite{EKPV3},
\cite{EKPV4}, \cite{EKPV5} and the references therein.

\section{Some recent quantitative unique continuation theorems}

Here I will discuss some quantitative unique continuation theorems for elliptic, parabolic, and dispersive equations. I will start by describing the elliptic situation. This arose as a key step in the work of \cite{BK} which proved Anderson localization at the bottom of the spectrum for the continuous Bernoulli model in higher dimensions, a question originating in Anderson's paper \cite{A}. Briefly, this says the following: consider a random Schr\"odinger operator on $\R^n$, of the form $H_\epsilon=-\triangle+V_\epsilon$, where
\begin{displaymath}
V_\epsilon(x)=\sum_{j\in\mathbb{Z}^n}\epsilon_j\phi(x-j),\quad\phi\in C_0^\infty(B(0,1/10)),\quad 0\leq\phi\leq1
\end{displaymath}
and $\epsilon_j\in\{0,1\}$ are independent. It is not difficult to see that $\inf\spec H_\epsilon=0$ a.s. . In this context, Anderson localization means that for energies $E$ near the bottom of the spectrum (i.e. $0<E<\delta$) $H_\epsilon$ has pure point spectrum, with exponentially decaying eigenfunctions, a.s. . When $V_\epsilon$ has  a continuous site distribution ($\epsilon_j\in[0,1]$) this has been understood for some time (\cite{GMP} $n=1$, \cite{FS} $n>1$). For the Anderson-Bernoulli model this was known for $n=1$ (\cite{CKM}; \cite{SVW}), but not in higher dimensions. We now have:
\begin{thm}[\cite{BK}]
There exists $\delta>0$ s.t. for $0<E<\delta$, $H_\epsilon$ displays Anderson localization a.s., $n\geq1$.
\end{thm}

In establishing this result we were lead to the following deterministic quantitative unique continuation theorem:
Suppose that $u$ is a solution to $\triangle u+Vu=0$ in $\R^n$, where $|V|\leq1$, and $|u|\leq C_0$, $u(0)=1$. For $R$ large, define 
\begin{displaymath}
M(R)=\inf_{|x_0|=R}\sup_{B(x_0,1)}|u(x)|.
\end{displaymath}
Note that by unique continuation, $\sup_{B(x_0,1)}|u(x)|>0$. How small can $M(R)$ be?
\begin{thm}[\cite{BK}]
\begin{displaymath}
M(R)\geq C\exp(-CR^{4/3}\log R).
\end{displaymath}
\end{thm}

\begin{rem}

In order for our argument to give the desired application to Anderson localization for the Bernoulli model, we would need an estimate of the form $M(R)\geq C\exp(-CR^\beta)$, with $\beta<\frac{1+\sqrt{3}}{2}\approx1.35$. Note that $4/3=1.333\ldots$.
\end{rem}

As it turns out, this is a quantitative version of a conjecture of E.M. Landis. He conjectured (late 60's) that if $\triangle u+Vu=0$ in 
$\R^n$, where $|V|\leq1$,  $|u|\leq C_0$, and $\displaystyle|u(x)|\leq C\exp(-C|x|^{1+})$, then $u\equiv0$. This conjecture of Landis was disproved by Meshkov (\cite{M}), who constructed such a $V$, $u\not\equiv0$, with $\displaystyle|u(x)|\leq C\exp(-C|x|^{4/3})$. This example also shows the sharpness of our lower bound on $M(R)$. One should note however that in Meshkov's example $u$, $V$ are complex valued.

Our proof uses a rescaling procedure, combined with well-known Carleman estimates.

\paragraph{\textbf{Q:}} Can $4/3$ be improved to $1$ in our lower bound for $M(R)$ for real valued $u,V$?

Let us now turn our attention to parabolic equations. Thus, consider solutions to
\begin{displaymath}
\partial_t u -\triangle u+W(x,t)\cdot\nabla u+V(x,t)u=0
\end{displaymath}
in $\R^n\times(0,1]$, with $|W|\leq N$, $|V|\leq M$. Then, as is well-known, the following backward uniqueness result holds: 
If $|u(x,t)|\leq C_0$ and $u(x,1)\equiv0$, then $u\equiv0$ (see \cite{LO}). This result has been extended by Escauriaza-Seregin-\v Sver\'ak (\cite{ESS}) who showed that it is enough to assume that $u$ is a solution on $\R^n_+\times(0,1]$, where 
$\R^n_+=\{x=(x',x_n):x_n>0\}$, without any assumption on $u\arrowvert_{\partial\R^n_+\times[0,1]}$. This was a crucial ingredient in their proof that weak (Leray--Hopf) solutions of the Navier--Stokes system in $\R^3\times[0,1)$, which have uniformly bounded $L_x^3$ norm are regular and unique. In 1974, Landis--Oleinik, \cite{LO}, in parallel to Landis' conjecture for elliptic equations mentioned earlier, formulated the following conjecture: Let $u$ be as in the backward uniqueness situation mentioned above. Assume that, instead of $u(x,1)\equiv0$, we assume that $|u(x,1)|\leq C\exp(-C|x|^{2+\epsilon})$, for some $\epsilon>0$. Is then $u\equiv0$? Clearly, the exponent 2 is  optimal here.

\begin{thm}[\cite{EKPV1}] 
The Landis--Oleinik conjecture holds. More precisely, if $||u(\cdot,1)||_{L^2(B(0,1))}\geq\delta$, there exists 
$R_0=R_0(\delta,M,N,n)>0$ s.t. for $|y|\geq R_0$, we have
\begin{displaymath}
||u(\cdot,1)||_{L^2(B(0,1))}\geq C\exp(-C|y|^2\log|y|).
\end{displaymath}
Moreover, an analogous result holds for $u$ only defined in $\R^n_+\times(0,1]$.
\end{thm}

The proof of this result uses space-time rescalings and parabolic Carleman estimates, in the spirit of the elliptic case. It holds for both real and complex solutions. We hope that this result will prove useful in control theory.

We now turn our attention to dispersive equations. Ler us consider non-linear Schr\"odinger equations of the form
\begin{displaymath}
i\partial_t u+\triangle u+F(u,\overline u)u=0,\quad\mathrm{in}\;\R^n\times[0,1],
\end{displaymath}
for suitable non-linearity $F$, and let us try to understand  what (if any) is the analog of the parabolic result we have just explained. The first obstacle is that the Schr\"odinger equations are time reversible and so ``backward'' makes no sense here. As is usual for uniqueness questions, we consider linear Schr\"odinger equations of the form
\begin{displaymath}
i\partial_t u+\triangle u+Vu=0,\quad\mathrm{in}\;\R^n\times[0,1],
\end{displaymath}
and deal with suitable $V(x,t)$ so that we can, in the end, set $$V(x,t)=F(u(x,t),\overline u(x,t)).$$ In order to motivate our work, I will first recall the following version of Heisenberg's uncertainty principle, due to Hardy, \cite{SS}: if $f:\R\to\mathbb{C}$, and we have 
$f(x)=\mathcal{O}(e^{-\pi Ax^2})$ and $\hat f(\xi)=\mathcal{O}(e^{-\pi B\xi^2})$,
 $A,B>0$, if $A\cdot B>1$, then $f\equiv0$. For instance, if 
\begin{displaymath}
|f(x)|\leq C_\epsilon\exp(-C_\epsilon |x|^{2+\epsilon}),\quad |\hat f(\xi)|\leq C_\epsilon\exp(-C_\epsilon |\xi|^{2+\epsilon}), 
\end{displaymath}
then 
$f\equiv0$. This can easily be translated into an equivalent formulation for solutions to the free Schr\"odinger equation. For, if $v$ solves 
\begin{displaymath}
i\partial_tv+\partial_x^2v=0\quad \text{in } \R\times[0,1],
\end{displaymath}
 with $v(x,0)=v_0(x)$, then 
\begin{displaymath}
v(x,t)=\frac{C}{\sqrt{t}}\int e^{i|x-y|/4t}v_0(y)dy,
\end{displaymath}
so that
\begin{displaymath}
v(x,1)=Ce^{i|x|^2/4}\int e^{-ixy/2}e^{i|y|^2/4}v_0(y)dy.
\end{displaymath}
If we then apply the corollary to Hardy's uncertainty principle to $f(y)=e^{iy^2/4}v_0(y)$, we se that if 
\begin{displaymath}
|v(x,0)|\leq C_\epsilon\exp(-C_\epsilon |x|^{2+\epsilon})\quad\text{and}\quad|v(x,1)|\leq C_\epsilon\exp(-C_\epsilon |x|^{2+\epsilon}),
\end{displaymath}
we must have $v\equiv0$. Thus, for time-reversible equations, the analog of backward uniqueness should be ``uniqueness from behavior at two different times''. Thus, we are interested in such results with ``data eventually 0'' or even with ``decaying very fast data''. This kind of uniqueness question for ``data eventually 0'' has been studied for some time. For the 1-d cubic Schr\"odinger equation
\begin{displaymath}
i\partial_tu+\partial_x^2u\mp|u|^2u=0\quad\text{in}\quad\R\times[0,1],
\end{displaymath}
B.Y. Zhang (\cite{Z2}) showed that if $u\equiv0$ on $(-\infty,a]\times\{0,1\}$ or on $[a,+\infty)\times\{0,1\}$, $a\in\R$, then $u\equiv0$ on 
$\R\times[0,1]$. His proof used inverse scattering, a non-linear Fourier transform, and analyticity. In 2002, \cite{KPV3} did away with scattering and analyticity, proving corresponding results for solutions to 
\begin{displaymath}
i\partial_tu+\triangle u+V(x,t)u=0\quad\text{in}\quad\R^n\times[0,1],\quad n\geq1.
\end{displaymath}

\begin{thm}[\cite{KPV3}]\label{thm15}
If $V\in L_t^1L_x^\infty\cap L_{loc}^\infty$ and 
\begin{displaymath}
||V||_{L_t^1L^\infty(|x|>R)}\xrightarrow[R\to\infty]{}0
\end{displaymath}
and there exists a strictly convex cone $\Gamma\subset\R^n$ and a $y_0\in\R^n$ such that
\begin{displaymath}
\supp u(\cdot,0)\subset y_0+\Gamma,\quad\quad\supp u(\cdot,1)\subset y_0+\Gamma,
\end{displaymath}
then we must have $u\equiv0$ on $\R^n\times[0,1]$.
\end{thm}
Clearly, taking $V(x,t)=|u|^2(x,t)$, we recover Zhang's result mentioned above.
This was extended by \cite{IK} who considered more general potentials $V$ and the case when $\Gamma=\R^n_+$. For instance, if 
$V\in L^{\frac{n+2}{2}}(\R^n\times[0,1])$ or even $V\in L_t^pL_x^q(\R^n\times[0,1])$ with $2/p+n/q\leq2$, $1<p<\infty$ 
($n=1$, $1<p<2$) or $V\in C([0,1];L^{n/2}(\R^n))$ $n\geq3$, the result holds with $\Gamma$ a half-plane. Our extension of Hardy's uncertainty principle, to this context, now is:

\begin{thm}[\cite{EKPV2}]\label{thm16}
Let $u$ be a solution of 
\begin{displaymath}
i\partial_t u+\triangle u+Vu=0,\quad\mathrm{in}\;\R^n\times[0,1].
\end{displaymath}
Assume that $V\in L^\infty(\R^n\times[0,1])$, $\nabla_x V\in L_t^1([0,1];L^\infty_x(\R^n))$ and 
\begin{displaymath}
\lim_{R\uparrow\infty}||V||_{L_t^1L^\infty(|x|>R)}=0.
\end{displaymath}
 If there exists $\alpha>2$, $a>0$, such that 
$u(\cdot,0),u(\cdot,1)\in H^1(e^{a|x|^\alpha}dx)$, then  $u\equiv0$.
\end{thm}

It is conjectured that Theorem \ref{thm16} remains valid assuming only that $u$, $\nabla u$ at times $0$, $1$ are in 
$L^2\left((y_0+\Gamma),e^{a|x|^\alpha}dx\right)$, with $y_0+\Gamma$ as in Theorem \ref{thm15}. This extension of Theorem \ref{thm16} would clearly imply Theorem \ref{thm15}.

Let me sketch the prof of this result. Our starting point is:

\begin{lem}[\cite{KPV3}]\label{lem17}
$\exists\epsilon>0$ s.t. if $||V||_{L_t^1L_x^\infty}\leq\epsilon$ and $u$ solves 
\begin{displaymath}
i\partial_t u+\triangle u+Vu=H,\quad\mathrm{in}\;\R^n\times[0,1],
\end{displaymath}
and $u_0(x)=u(x,0)$, $u_1(x)=u(x,1)$ belong to $L^2(e^{2\beta x_1}dx)\cap L^2(dx)$
and $$H\in L_t^1(L^2(e^{2\beta x_1}dx)\cap L^2(dx)),$$ then $$u\in C([0,1]; L^2(e^{2\beta x_1}dx))$$ and
\begin{multline*}
\sup_{0\leq t\leq1}||u(\cdot,t)||_{L^2(e^{2\beta x_1}dx)}\leq\\\leq
C\left\{||u_0||_{L^2(e^{2\beta x_1}dx)}+||u_1||_{L^2(e^{2\beta x_1}dx)}+||H||_{L_t^1L^2(e^{2\beta x_1}dx)}\right\}
\end{multline*}
with $C$ independent of $\beta$.
\end{lem}

This is a delicate lemma. If we a priori knew that $u\in C([0,1]; L^2(e^{2\beta x_1}dx))$, a variant of the energy method, splitting frequencies into $\xi_1>0$, $\xi_1<0$, gives the result. But, since we are not free to prescribe both $u_0$, $u_1$, we cannot use a priori estimates. This is instead accomplished by ``truncating'' the weight $2\beta x_1$ and introducing an extra parameter.

Or next step is to deduce, from  Lemma \ref{lem17}, further weighted estimates:

\begin{cor}\label{cor18}
Assume that we are under the hypothesis of  Lemma \ref{lem17} and for some $a>0$, $\alpha>1$, 
\begin{displaymath}
u_0,u_1\in L^2(e^{a|x|^\alpha}dx),\quad
H\in L_t^1L_x^2(e^{a|x|^\alpha}dx). 
\end{displaymath}
Then $\exists C_\alpha>0$, $b>0$ s.t.
\begin{displaymath}
\sup_{0<t<1}\int_{|x|>C_\alpha}e^{b|x|^\alpha}|u(x,t)|^2dx<\infty.
\end{displaymath}
\end{cor}

Idea for the proof of Corollary \ref{cor18}: Multiply $u$ by $\eta_R(x)=\eta(x/R)$, $\eta\equiv0$ for $|x|\leq1$, $\eta\equiv1$ for $|x|>2$. We apply Lemma \ref{lem17} to $u_R(x,t)=\eta(x/R)u(x,t)$, with $\beta=\gamma R^{\alpha-1}$, for suitable $\gamma$ and the corollary follows.

The next step of the proof is to deduce lower bounds for $L^2$ space-time integrals, in analogy with the elliptic and parabolic arguments. These are ``quantitative''.

\begin{thm}\label{thm19}
Let $u$ solve
$
i\partial_t u+\triangle u+Vu=0,$ $x\in\R^n,t\in[0,1].
$
Assume that
\begin{displaymath}
\int_0^1\int_{\R^n}|u|^2+|\nabla u|^2\leq A,\quad\text{and that}\quad
\int_{\frac{1}{2}-\frac{1}{8}}^{\frac{1}{2}+\frac{1}{8}}\int_{|x|<1}|u|^2dxdt\geq1,
\end{displaymath}
with $||V||_\infty\leq L$. Then there exists $R_0=R_0(A,L,n)>0$ and $c_n$ s.t. if $R>R_0$
\begin{displaymath}
\delta(R)=\left(\int_0^1\int_{R-1\leq|x|\leq R}(|u|^2+|\nabla u|^2)dxdt\right)^\frac{1}{2}\geq c_ne^{-c_nR^2}.
\end{displaymath}
\end{thm}

Clearly, Corollary \ref{cor18} applied to $u$, $\nabla u$, combined with Theorem \ref{thm19} yield our version of Hardy's uncertainty principle.

In order to prove  Theorem \ref{thm19}, we use a Carleman estimate which is a variant of one due to V. Isakov \cite{I}.

\begin{lem}[\cite{EKPV2}]\label{lem110}
Assume that $R>0$ and $\phi:[0,1]\to\R$ is a smooth real function. Then, there exists 
$C=C(n,||\phi'||_\infty,||\phi''||_\infty)>0$ s.t.
\begin{displaymath}
\frac{\alpha^{3/2}}{R^2}\left\lVert e^{\alpha\left|\frac{x}{R}+\phi(t)\vec e_1\right|^2}g\right\rVert_{L^2}\leq
C_n\left\lVert e^{\alpha\left|\frac{x}{R}+\phi(t)\vec e_1\right|^2}(i\partial_t+\triangle)g\right\rVert_{L^2},
\end{displaymath}
for all $\alpha>C_nR^2$, $g\in C_0^\infty(\R^{n+1})$ s.t. $\supp g\subset\{(x,t):\left|\frac{x}{R}+\phi(t)\vec e_1\right|\geq1\}$.
\end{lem}

The proof of Lemma \ref{lem110} follows by conjugating the operator $(i\partial_t+\triangle)$ with the weight 
$\exp\left(\alpha\left|\frac{x}{R}+\phi(t)\vec e_1\right|^2\right)$, and splitting the resulting operator into a Hermitian and an 
anti-Hermitian part. Then, the commutator between the two parts is positive, for $g$ with the support property above and 
$\alpha\geq C_nR^2$.

In order to use Lemma \ref{lem110} to prove Theorem \ref{thm19}, we choose $\theta_R,\theta\in C^\infty(\R^n)$, 
$\phi\in C_0^\infty([0,1])$ so that $\theta_R(x)=1$ if $|x|<R-1$, $\theta_R(x)=0$, $|x|\geq R$;
$\theta(x)\equiv0$ if $|x|<1$, $\theta(x)\equiv1$, when $|x|\geq 2$;
$0\leq \phi\leq3$, with $\phi\equiv3$ on $\left[\frac{1}{2}-\frac{1}{8},\frac{1}{2}+\frac{1}{8}\right]$
and $\phi\equiv0$ on $[0,1/4]\cup[3/4,1]$. We apply  Lemma \ref{lem110} to 
$g(x,t)=\theta_R(x)\cdot\theta\left(\frac{x}{R}+\phi(t)\vec e_1\right)u(x,t)$, $\alpha\approx R^2$, to obtain, after some manipulations, the desired result.

We next turn our attention to corresponding results for the KdV equations. In \cite{Z1} it is proved that if 
\begin{displaymath}
\partial_tu+\partial_x^3u+u\partial_xu=0,\quad \text{in } \R\times[0,1],
\end{displaymath}
 and $u_0(x)=u(x,0)$, $u_1(x)=u(x,1)$ are supported in $(a,+\infty)$
or in $(-\infty,a)$, then $u\equiv0$. This was later extended by \cite{KPV1}, \cite{KPV2}, who also showed that if $v_1$, $v_2$ are solutions of 
\begin{displaymath}
\partial_tv+\partial_x^3v+v^k\partial_xv=0,\quad k\geq1,
\end{displaymath}
 and $u_0=v_1(x,0)-v_2(x,0)$, $u_1=v_1(x,1)-v_2(x,1)$ are supported in $(a,+\infty)$
or in $(-\infty,a)$, then $v_1\equiv v_2$. 

Further results are due to L. Robbiano (\cite{R}). He considered $u$ a solution to 
\begin{equation}\label{R}
\partial_tu+\partial_x^3u+a_2(x,t)\partial_x^2u+a_1(x,t)\partial_xu+a_0(x,t)u=0
\end{equation}
with coefficients $a_j$ in suitable function spaces. He showed that, if $u(x,0)=0$, $x\in(b,\infty)$ some $b$, and $\exists C_1,C_2>0$
s.t. 
\begin{displaymath}
|\partial_x^ju(x,t)|\leq C_1\exp(-C_2x^\alpha),\quad (x,t)\in(b,\infty)\times[0,1] 
\end{displaymath}
for some $\alpha>9/4$, then $u\equiv0$.

On the other hand, the Airy function $$A_i(x)=\int e^{2\pi ix\xi+\xi^3i}d\xi$$ is the fundamental solution for
$\partial_tu+\partial^3_xu=0$, and verifies
$$|A_i(x)|\leq C(1+x_-)^{-1/4}\exp(-Cx_+^{3/2}).$$ We now have

\begin{thm}[\cite{EKPV3}]\label{thm111}
If $u$ is a solution of \eqref{R} on $\R\times[0,1]$ such that $u(x,0), u(x,1)\in H^1(e^{ax_+^{3/2}}dx)$ for any $a>0$, and $a_j$ belong to suitable function spaces, then $u\equiv0$
\end{thm}

This is clearly optimal for $\partial_t u+\partial_x^3u=0$. The same result holds for $e^{ax_-^{3/2}}dx$. The proof of this theorem also has two steps, one consisting of upper bounds, the other of lower bounds. The second step follows closely that used for Schr\"odinger operators, but the upper bounds can no longer be obtained by any variant of the energy estimates. These are now replaced by suitable ``dispersive Carleman estimates''. A typical application of Theorem \ref{thm111} is:

\begin{thm}[\cite{EKPV3}]\label{thm112}
Let $$u_1, u_2\in C([0,1];H^3(\R))\cap L^2(|x|^2dx),$$ solve 
\begin{displaymath}
\partial_tu+\partial_x^3u+u^k\partial_xu=0\quad \text{on } \R\times[0,1]. 
\end{displaymath}
Assume that
$$u_1(\cdot,0)-u_2(\cdot,0),u_1(\cdot,1)-u_2(\cdot,1)\in H^1(e^{ax_+^{3/2}}dx)$$ for any $a>0$. Then $u_1\equiv u_2$.
\end{thm}

Finally, we end with a result that shows that this result is sharp, even for the non-linear problem.

\begin{thm}[\cite{EKPV3}]\label{thm113}
There exists $u\not\equiv0$, a solution of 
\begin{displaymath}
\partial_tu+\partial_x^3u+u^k\partial_xu=0\quad \text{in } \R\times[0,1] 
\end{displaymath} s.t.
\begin{displaymath}
|u(x,0)|+|u(x,1)|\leq C\exp(-Cx_+^{3/2}).
\end{displaymath}
\end{thm}

\section{Convexity properties of Gaussian means of solutions to Schr\"odinger equations}

As mentioned before, \cite{EKPV2} proved that if $u\in C([0,1];H^1(\R^n))$ solves
\begin{displaymath}
\left\{\begin{array}{ll}
i\partial_tu+\triangle u+V(x,t)u=0&\quad\text{in }\R^n\times[0,1]\\
u(0)=u_0& \\
u(1)=u_1
\end{array}\right.
\end{displaymath}
and $u_i\in L^2(e^{a|x|^\theta}dx)$ for some $a>0$, $\theta>1$, then $\exists C_\theta>0$, $b>0$ s.t.
\begin{displaymath}
\sup_{0<t<1}\int_{|x|>C_\theta}e^{b|x|^\theta}|u(x,t)|^2dx<\infty
\end{displaymath}
when the (complex) potential verifies $||V||_{L_t^1L_x^\infty}\leq\epsilon$, $\epsilon=\epsilon_n$. We will next re-examine this result and precise it, in the case $\theta=2$. We will first deal with potentials $V=V(x)$, $V$ real valued; $||V||_\infty\leq M_1$. We will consider $u\in C([0,1];L^2(\R^n))$ which verifies
\begin{displaymath}
\partial_tu=i(\triangle u+Vu)\quad \text{in } \R^n\times[0,1].
\end{displaymath}
 We will assume that there exist positive numbers $\alpha$ and $\beta$ such that
$||e^{|x|^2\!/\!\beta^2}\!\! u(0)||$, $||e^{|x|^2/\alpha^2}u(1)||$ are finite. Here and in the sequel $||\cdot||$ denotes the $L^2$ norm in $x$. Then 
$$\left\lVert e^{|x|^2/(\alpha t+(1-t)\beta)^2}u(t)\right\rVert^{\alpha t+(1-t)\beta}$$ is ``logarithmically convex'' in $[0,1]$, i.e. 

\begin{thm}[\cite{EKPV5}]\label{thm1}
There exists $N=N(\alpha,\beta)$ so that for $0<s<1$ we have
\begin{multline*}
\left\lVert e^{|x|^2/(\alpha s+(1-s)\beta)^2}u(s)\right\rVert\leq 
e^{N(M_1+M_1^2)}\\\times\left\lVert e^{|x|^2/\beta^2}u(0)\right\rVert^{\beta(1-s)/(\alpha s+(1-s)\beta)}
\left\lVert e^{|x|^2/\alpha^2}u(1)\right\rVert^{\alpha s/(\alpha s+(1-s)\beta)}.
\end{multline*}
Moreover (``smoothing effect'')
\begin{multline*}
\left\lVert \sqrt{t(1-t)}e^{|x|^2/(\alpha t+(1-t)\beta)^2}\nabla u(t)\right\rVert_{L^2(\R^n\times[0,1])}\leq\\
Ne^{N(M_1+M_1^2)}\left[||e^{|x|^2/\beta^2}u(0)||+||e^{|x|^2/\alpha^2}u(1)||\right].
\end{multline*}
\end{thm}

Note that when $\alpha=\beta$, we have $\alpha t+(1-t)\beta=\alpha$ and this gives the precise version (for this case) of the \cite{EKPV2} result. We start with the sketch of the proof in the case $\alpha=\beta$. It turns out that a formal argument giving the proof is not too dificult, but a rigurous justification is tricky. This is an important fact, since, as we will see, the formal arguments actually can lead to false results. To justify the interest of the case $\alpha\not=\beta$, consider the case $V\equiv0$, i.e. the free particle. Then, if $u_0=u(0)$,
\begin{multline*}
u(x,t)=(e^{-i|\xi|^2t}\hat u_0)\,{\check{ }}=
\int_{\R^n}\frac{e^{i|x-y|^2/4t}}{(4\pi it)^{n/2}}u_0(y)dy=\\=
\frac{e^{i|x|^2/4t}}{(4\pi it)^{n/2}}\int e^{-2ix\cdot\xi/4t}e^{i|y|^2/4t}u_0(y)dy=
\frac{e^{i|x|^2/4t}}{(2it)^{n/2}}\widehat{\left(e^{i|\cdot|^2/4t}u_0\right)}(x/2t),
\end{multline*}
so that, with $c_t=(2it)^{n/2}$,
\begin{displaymath}
c_te^{-i|x|^2/4t}u(x,t)=(e^{i|\cdot|^2/4t}u_0)\,\hat{ }\,(x/2t).
\end{displaymath} 
In this context the Hardy uncertainty principle says that if $$u(0)\in L^2(e^{2|x|^2/\beta^2}dx),\quad 
u(1)\in L^2(e^{2|x|^2/\alpha^2}dx),$$ with $\alpha\beta\leq4,$ then $u_0\equiv0$ and $4$ is sharp.

\begin{kcl}[abstract]
Let $S$ be a symmetric operator, $A$ an anti-symmetric one (possibly both depending on $t$), $F$ a positive function, $f(x,t)$ a ``reasonable function''. Let $H(t)=(f,f)$, $D(t)=(Sf,f)$, $\partial_t S=S_t$ and $N(t)=D(t)/H(t)$ (the ``frequency function''). Then
\begin{itemize}
\item[i)] 
$\partial_t^2H=2\partial_t\re(\partial_tf-Sf-Af,f)+\\+2(S_tf+[S,A]f,f)+||\partial_tf-Af-Sf||^2-||\partial_tf-Af-Sf||^2$\\
and
\item[ii)] $\dot N(t)\geq(S_tf+[S,A]f,f)/H-||\partial_tf-Af-Sf||^2/(2H)$

\item[iii)] Moreover, if 
\begin{displaymath}
|\partial_tf-Af-Sf|\leq M_1|f|+F\quad \text{in }  \R^n\times[0,1],\quad S_t+[S,A]\geq-M_0
\end{displaymath}
 and 
\begin{displaymath}
\displaystyle M_2=sup_{0\leq t\leq1}||F(t)||/||f(t)||<\infty,
\end{displaymath}
 then $H(t)$ is ``logarithmically convex'' in $[0,1]$ and
$$
H(t)\leq e^{N(M_0+M_1+M_1^2+M_2+M_2^2)}H(0)^{1-t}H(1)^t.
$$
\end{itemize}
\end{kcl}

\begin{proof}
\begin{displaymath}
\dot H(t)=2\re(\partial_tf,f)=2\re(\partial_tf-Sf-Af,f)+2(Sf,f),
\end{displaymath}
so
\begin{equation}\label{hdt}
\dot H(t)=2\re(\partial_tf-Sf-Af,f)+2D(t).
\end{equation}
Also,
\begin{displaymath}
\dot H(t)=\re(\partial_tf+Sf,f)+\re(\partial_tf-Sf,f),
\end{displaymath}
\begin{displaymath}
D(t)=\frac{1}{2}\re(\partial_tf+Sf,f)-\frac{1}{2}\re(\partial_tf-Sf,f).
\end{displaymath}
Multiplying
\begin{equation}\label{HD}
\dot H(t)D(t)=\frac{1}{2}\re(\partial_tf+Sf,f)^2-\frac{1}{2}\re(\partial_tf-Sf,f)^2.
\end{equation}
Adding an anti-symmetric part does not change the real parts, so
\begin{equation}\label{HD2}
\dot H(t)D(t)=\frac{1}{2}\re(\partial_tf+Sf-Af,f)^2-\frac{1}{2}\re(\partial_tf-Sf-Af,f)^2.
\end{equation}
Differentiating $D(t)$, we get
\begin{multline}\label{FF}
\dot D(t)=(S_tf,f)+(S\partial_tf,f)+(Sf,\partial_tf)
=(S_tf,f)+2\re(\partial_tf,Sf)=\\
=(S_tf+[S,A]f,f)+2\re(\partial_tf-Af,Sf)=\\
=(S_tf,[S,A]f,f)+\frac{1}{2}||\partial_tf-Af+Sf||^2-\frac{1}{2}||\partial_tf-Af-Sf||^2
\end{multline}
by polarization. This and \eqref{hdt} gives i). Next,
\begin{multline*}
\dot N(t)= (S_tf+[S,A]f,f)/H+\\
+\frac{1}{2}\left[||\partial_tf-Af+Sf||^2||f||^2-(\re(\partial_tf-Af+Sf,f))^2\right]/H^2+\\
+\frac{1}{2}\left[(\re(\partial_tf-Af-Sf,f))^2-||\partial_tf-Af-Sf||^2||f||^2\right]/H^2
\end{multline*}
follows from \eqref{HD2}. Now, the second line is non-negative (Cauchy-Schwartz), $$(\re(\partial_tf-Af-Sf,f))^2\geq0,$$ so ii) follows.

When we are in the situation of iii), $$\dot N(t)\geq-M_0+M_1^2+M_2^2,$$ so that \eqref{hdt} now gives 
$\partial_t[\log H(t)]=\mathcal{O}(1)+2N(t)$. If $G'(t)=\mathcal{O}(1)$, $G(0)=0$, we get $\partial_t[\log H(t)-G(t)]=2N(t)$, so that
$$\partial_t^2[\log H(t)-G(t)]\geq-(M_0+M_1^2+M_2^2),$$ so that
\begin{displaymath}
\partial_t^2\left[\log H(t)-G(t) +(M_0+M_1^2+M_2^2)t^2/2\right]\geq0
\end{displaymath}
which gives the desired ``log convexity''.
\end{proof}

\paragraph{\textbf{Sketch of Proof ($\alpha=\beta=\gamma$).}} 

Let us now indicate how the ``formal argument'' for the first part of Theorem \ref{thm1} would follow, when $\alpha=\beta=\gamma$. Suppose now (for later use) that
\begin{displaymath}
\partial _tu=(a+ib)(\triangle u+V(x,t)u+F(x,t))\quad\text{in }\R^n\times[0,1],
\end{displaymath}
$a\geq0$, $||e^{\gamma|x|^2}u(0)||<\infty$, $||e^{\gamma|x|^2}u(1)||<\infty$, $\sup_{[0,1]}||e^{\gamma|x|^2}F(x,t)||/||u(t)||=M_2$,
$V$ is complex valued, $||V||_\infty\leq M_1$. Let $f=e^{\gamma\phi}u$, where $\phi(x,t)$ is to be chosen. Then, $f$ verifies
\begin{displaymath}
\partial_tf=Sf+Af+(a+ib)(Vf+e^{\gamma\phi}F),
\end{displaymath}
where
\begin{displaymath}
S=a(\triangle+\gamma^2|\nabla\phi|^2)-ib\gamma(2\nabla\phi\cdot\nabla+\triangle\phi)+\gamma\partial_t\phi,
\end{displaymath}
\begin{displaymath}
A=ib(\triangle+\gamma^2|\nabla\phi|^2)-a\gamma(2\nabla\phi\cdot\nabla+\triangle\phi)
\end{displaymath}
are symmetric, anti-symmetric. When $\phi(x,t)=|x|^2$ we obtain
\begin{displaymath}
S_t+[S,A]=-\gamma(a^2+b^2)[8\triangle-32\gamma^2|x|^2],
\end{displaymath}
so that $$S_t+[S,A]\geq0,\quad |\partial_tf-Sf-Af|\leq\sqrt{a^2+b^2}(M_1|f|+e^{\gamma|x|^2}|F|)$$ and the Lemma ``gives'' the (formal) ``log convexity'' result.

We need to have an argument which gives us the required smoothness and decay to justify the formal argument. Before doing that, we give the ``formal'' argument for the smoothing estimate: first note that integration by parts shows that
\begin{displaymath}
\int|\nabla f|^2+4\gamma^2|x|^2|f|^2=\int e^{2\gamma|x|^2}(|\nabla u|^2-2n\gamma|u|^2)dx
\end{displaymath}
where $f=e^{\gamma|x|^2}u$. Also, since $n=\nabla\cdot x$, integration by parts and Cauchy-Schwartz give
\begin{displaymath}
\int|\nabla f|^2+4\gamma^2|x|^2|f|^2\geq2\gamma n\int|f|^2dx.
\end{displaymath}
Adding we obtain
\begin{equation}\label{x}
2\left(\int|\nabla f|^2+4\gamma^2|x|^2|f|^2\right)\geq\int e^{2\gamma|x|^2}|\nabla u|^2dx.
\end{equation}
Recall
\begin{multline*}
\partial_t^2H(t)=2\partial_t\re(\partial_t f-Sf-Af,f)+2(S_tf+[S,A]f,f)+\\+
||\partial_tf+Sf-Af||^2-||\partial_tf-Af-Sf||^2\geq\\\geq
2\partial_t\re(\partial_t f-Sf-Af,f)-||\partial_tf-Af-Sf||^2+2(S_tf+[S,A]f,f).
\end{multline*}
Multiply by $t(1-t)$ and integrate by parts to obtain
\begin{multline*}
2\int_0^1 t(1-t) (S_tf+[S,A]f,f)dt+2\int_0^1 H(t)dt\leq H(1)+H(0)+\\+2\int_0^1(1-2t)\re(\partial_t f-Sf-Af,f)+
\int_0^1t(1-t)||\partial_t f-Sf-Af||^2dt.
\end{multline*}
We now use 
\begin{displaymath}
S_t+[S,A]=-\gamma(a^2+b^2)[8\triangle-32\gamma^2|x|^2],
\end{displaymath}
\begin{displaymath}
|\partial_t f-Sf-Af|\leq\sqrt{a^2+b^2}\left(M_1|f|+e^{\gamma|x|^2}|F|\right),
\end{displaymath}
to obtain:
\begin{multline*}
16\gamma(a^2+b^2)\int_0^1\int t(1-t)\left\{|\nabla f|^2+4\gamma^2|x|^2|f|^2\right\}\leq\\
\leq\left[(NM_1^2+1)\sup_{[0,1]}\left\lVert e^{\gamma|x|^2}u(t)\right\rVert^2+\sup_{[0,1]}\left\lVert e^{\gamma|x|^2}F\right\rVert^2
\right](a^2+b^2).
\end{multline*}
Finally, $\nabla f= e^{\gamma|x|^2}(\nabla u+2xu\gamma)$, and \eqref{x} gives the bound: ($\gamma>0$)
\begin{multline*}
||\sqrt{t(1-t)}e^{\gamma|x|^2}\nabla u||_{L^2(\R^n\times[0,1])}+||\sqrt{t(1-t)}|x|e^{\gamma|x|^2} u||_{L^2(\R^n\times[0,1])}
\leq \\ \leq N\left[(1+\sqrt{M_1}+M_1)\sup_{[0,1]}\left\lVert e^{\gamma|x|^2}|u(t)|\right\rVert+
\sup_{[0,1]}\left\lVert e^{\gamma|x|^2}F\right\rVert_{L^2(\R^n\times[0,1])}\right].
\end{multline*}

How to justify the formal arguments? We first change $i(\triangle+V)$ by $(a+i)(\triangle+V)$, we change $|x|^2$ by $|x|^{2-\epsilon}$,
 $a>0$, $\epsilon>0$ and then pass to the limit. This can be justified when $V=V(x)$, real, bounded. This is how we proceed:

\begin{lem}[Energy method] Assume that $u\in L^\infty([0,1];L^2)\cap L^2([0,1];H^1)$ satisfies
\begin{displaymath}
\partial_t u=(a+ib)(\triangle u+V(x,t)u)+F(x,t)\quad \text{in }\R^n\times[0,1],
\end{displaymath}
$a>0$, $b\in\R$. Then, for $0\leq T\leq1$,
\begin{multline*}
e^{-M_T}\left\lVert e^{\gamma a|x|^2/(a+4\gamma(a^2+b^2)T)}u(T)\right\rVert\leq\\\leq
\left\lVert e^{\gamma|x|^2}u(0)\right\rVert+
\sqrt{a^2+b^2}\left\lVert e^{\gamma a|x|^2/(a+4\gamma(a^2+b^2)T)}F\right\rVert_{L^1([0,T];L^2)},
\end{multline*}
where $M_T=||a\re V-b\im V||_{L^1([0,T];L^\infty)}$.
\end{lem}
\begin{proof}
For $\phi$ real, to be chosen, $v=e^\phi u$, $v$ verifies 
\begin{displaymath}
\partial_tv=Sv+Av+(a+ib)e^\phi F\quad\text{in }\R^n\times(0,1],
\end{displaymath}
where $S=$sym, $A=$anti-sym, 
\begin{displaymath}
S=a(\triangle+|\nabla\phi|^2)-ib(2\nabla\phi\cdot\nabla+\triangle\phi)+(\partial_t\phi+a\re V-b\im V)
\end{displaymath}
and
\begin{displaymath}
A=ib(\triangle+|\nabla\phi|^2)-a(2\nabla\phi\cdot\nabla+\triangle\phi)+i(b\re V-a\im V).
\end{displaymath}
\begin{displaymath}
\partial_t||v||^2=2\re(Sv,v)+2\re((a+ib)e^\phi F,v)\quad\text{(formally)}.
\end{displaymath}
A (formal) integration by parts gives
\begin{multline*}
\re(Sv,v)=-a\int|\nabla v|^2+\int(a|\nabla\phi|^2+\partial_t\phi)|v|^2+\\+2b\im\int\overline v\nabla\phi\cdot\nabla v+
\int(a\re V-b\im V)|v|^2.
\end{multline*} 
Cauchy-Schwartz gives
\begin{displaymath}
\partial_t||v(t)||^2\leq 2||a\re V-b\im V||_\infty||v(t)||^2+2\sqrt{a^2+b^2}||e^\phi F(t)||\,||v(t)||
\end{displaymath}
when
\begin{displaymath}
\left(a+\frac{b^2}{a}\right)|\nabla\phi|^2+\partial_t\phi\leq0.
\end{displaymath}

When $\phi(x,t)=h(t)\phi(x)$, it suffices that 
\begin{displaymath}
h^2(t)\left(a+\frac{b^2}{a}\right)|\nabla\phi(x)|^2+h'(t)\phi(x)\leq0.
\end{displaymath}
Eventually, we choose $\phi(x)=|x|^2$. We then choose
\begin{displaymath}
\left\{\begin{array}{l}
h'(t)=-4\left(a+\frac{b^2}{a}\right)h^2(t)\\
h(0)=\gamma
\end{array}\right.,
\end{displaymath}
so that $h(t)=\gamma a/(a+4\gamma(a^2+b^2)t)$. To formalize the calculations, given $R >0$, set 
\begin{displaymath}
\phi_R(x)=\left\{\begin{array}{ll}
|x|^2&\quad|x|\leq R\\ R^2&\quad|x|\geq R
\end{array}\right.,
\end{displaymath}
choose a radial mollifier $\theta_\rho$ and set
\begin{displaymath}
\phi_{\rho, R}(x,t)=h(t)\theta_\rho\ast\phi_R(x),\quad v_{\rho,R}=e^{\phi_{\rho, R}}u.
\end{displaymath}
Then, $\theta_\rho\ast\phi_R\leq\theta_\rho\ast|x|^2=|x|^2+C(n)\rho^2$, and our inequality above holds uniformly in $\rho$ and $R$. We obtain the result for $v_{\rho,R}$, let $\rho\to0$, then $R\to\infty$, which gives the final estimate. Note that, for $a>0$, Gaussian decay at $t=0$ is preserved, with a loss.
\end{proof}

Next, we prove that if  $u\in L^\infty([0,1];L^2)\cap L^2([0,1];H^1)$ verifies 
\begin{displaymath}
\partial_t u=(a+ib)(\triangle u +V(x,t)u+F(x,t)),
\end{displaymath}
where $||V||_{L^\infty}\leq M_1$, $\sup_{[0,1]}||e^{\gamma|x|^2}F(t)||/||u(t)||=M_2<\infty$, and $||e^{\gamma|x|^2}u(0)||$,
$||e^{\gamma|x|^2}u(1)||$ are finite, we have a ``log convex'' estimate, uniformly in $a>0$, small. In fact, we now repeat the formal argument, but replace $\phi(x)=|x|^2$ by 
\begin{displaymath}
\phi_\epsilon(x)=\left\{\begin{array}{ll}
|x|^2\quad&|x|\leq1\\ & \\\frac{2|x|^{2-\epsilon}-\epsilon}{2-\epsilon}\quad&|x|\geq1\end{array}\right.
\end{displaymath}
and then by $\phi_{\epsilon,\rho}(x)=\theta_\rho\ast\phi_\epsilon$, where $\theta_\rho\in C_0^\infty$ is radial. We then have: 
$\phi_{\epsilon,\rho}\in C^{1,1}$, it is convex and grows at infinity slower that $|x|^{2-\epsilon}$ and 
$\phi_{\epsilon,\rho}\leq|x|^2+C(n)\rho^2$. By the ``energy estimate'', for $a>0$, $\epsilon>0$, $\rho>0$, our argument applies rigurously, since $u(0)e^{\gamma|x|^2}\in L^2 \Rightarrow$ $0<t<1$, $u(t)e^{\gamma|x|^{2-\epsilon}}\in L^2$, and for a $t$ independent $\phi$,
\begin{displaymath}
S_t+[S,A]=-\gamma(a^2+b^2)\left[4\nabla\cdot(D^2\phi\nabla)-4\gamma^2D^2\phi\nabla\phi\cdot\nabla\phi+\triangle^2\phi\right].
\end{displaymath}
One can see that $||\triangle^2\phi_{\epsilon,\rho}||_\infty\leq C(n,\rho)\epsilon$, which gives the desired log convexity when 
$\epsilon\to0$, then $\rho\to0$, for $a>0$. Once the log convexity holds, for $a>0$ again, the ``local smoothing'' argument applies. The conclusion of these considerations is:
\begin{lem}
Assume that $u\in L^\infty([0,1];L^2(\R^n))\cap L^2([0,1];H^1)$ verifies
\begin{displaymath}
\partial_t u=(a+ib)(\triangle u +V(x,t)u+F(x,t)),\quad\text{in }\R^n\times[0,1],
\end{displaymath}
$\gamma>0$ where $a>0$, $b\in\R$, $||V||_\infty\leq M_1$. Then, $\exists N_\gamma$ s.t.
\begin{multline*}
\sup_{[0,1]}||e^{\gamma|x|^2}u(t)||\leq \\\leq e^{N_\gamma[(a^2+b^2)[M_1^2+M_2^2]+\sqrt{a^2+b^2}(M_1+M_2)]}||e^{\gamma|x|^2}u(0)||^{1-t}
||e^{\gamma|x|^2}u(1)||^t,
\end{multline*}
\begin{displaymath}
||\sqrt{t(1-t)}e^{\gamma|x|^2}u||_{L^2(\R^n\times[0,1])}\leq N_\gamma(1+M_1+M_2)\left\{\sup_{[0,1]}||e^{\gamma|x|^2}u(t)||\right\},
\end{displaymath}
where $M_2=\sup_{[0,1]}||e^{\gamma|x|^2}F(t)||/||u(t)||<\infty$.
\end{lem}

\paragraph{\textbf{Conclusion of the argument when $V(x,t)=V(x)$, real.}}
We now consider the Schr\"odinger operator $H=\triangle+V$, which is self-adjoint. We consider 
$u\in C([0,1];L^2)$ solving 
\begin{displaymath}
\partial_tu=i((\triangle+V)u)\quad\text{in }\R^n\times[0,1]
\end{displaymath}
and assume that $||e^{\gamma|x|^2}u(0)||<\infty$, $||e^{\gamma|x|^2}u(1)||<\infty$. 
From spectral theory, $u(t)=e^{iHt}u(0)$. Moreover, for $a>0$, consider the solution of 
\begin{displaymath}
\partial_tu_a=(a+i)((\triangle+V)u_a)\quad\text{in }\R^n\times[0,1],\;u_a(0)=u(0).
\end{displaymath}
We now have $$u_a(t)=e^{(a+i)tH}u(0)=e^{atH}e^{itH}u(0)=e^{atH}u(t).$$ Clearly
$$||e^{\gamma|x|^2}u_a(0)||=||e^{\gamma|x|^2}u(0)||.$$ Also, $u_a(1)=e^{aH}u(1)$. Recall, 
from the ``energy method'' that if
\begin{displaymath}
\left\{\begin{array}{l}
\partial_tv=a(\triangle+V)v\\v(0)=v_0
\end{array}\right.,\quad V \text{ real,}
\end{displaymath}
\begin{displaymath}
\left\lVert e^{\gamma a|x|^2/(a+4\gamma a^2)}v(1)\right\rVert\leq\exp(\tilde M_1)||e^{\gamma|x|^2}v_0||,
\end{displaymath}
where $\tilde M_1=||aV||_{L^1([0,1];L^\infty)}$. Now, if $v_0=u(1)$, then $v(1)=e^{aH}v_0=u_a(1)$,
so that 
\begin{displaymath}
\left\lVert e^{\gamma |x|^2/(1+4\gamma a)}u_a(1)\right\rVert\leq\exp(\tilde M_1)||e^{\gamma|x|^2}u(1)||.
\end{displaymath}
Let $\gamma_a=\gamma/(1+4\gamma a)$ and apply now our log-convexity result for $u_a$, $\gamma_a$.
We then obtain
\begin{multline*}
||e^{\gamma_a|x|^2}u_a(s)||\leq e^{NM_1}||e^{\gamma_a|x|^2}u_a(1)||^{1-s}||e^{\gamma_a|x|^2}u_a(0)||^s
\leq\\\leq e^{NM_1}\exp(\tilde M_1)||e^{\gamma|x|^2}u(1)||^{1-s}||e^{\gamma|x|^2}u(0)||^s.
\end{multline*}
We then let $a\to 0$ and obtain the ``log convexity'' bound. To obtain the ``local smoothing'' bound, 
we again use the $u_a$, let $a\downarrow0$. This establishes Theorem \ref{thm1} when $\alpha=\beta$.

\begin{rem}
Solutions so that $e^{\gamma|x|^2}u(0),\;e^{\gamma|x|^2}u(1)\in L^2$ certainly exist for some $\gamma$.
In fact, if $h\in L^2(e^{\epsilon|x|^2}dx)$ and $u_0=e^{\delta(\triangle+V)}h$, our ``energy method''
gives this for $u(t)=e^{it(\triangle+V)}u_0$, ($V=V(x)$). (We are indebted to R. Killip for this remark.) When $V\equiv0$, this characterizes such $u$! (see \cite{EKPV4}).
\end{rem}

\paragraph{\textbf{A misleading convexity argument:}}
Consider now $f=e^{a(t)|x|^2}u$, where $u$ solves the free Schr\"odinger equation
\begin{displaymath}
\partial_t u=i\triangle u\quad\text{in }\R\times[-1,1].
\end{displaymath}
Then, $f$ verifies $$\partial_tf=Sf+Af,$$ $$S=-4ia(x\partial_x+\frac{1}{2})+a'x^2,\quad
A=i(\partial_x^2+4a^2x^2).$$ In this case we have
\begin{displaymath}
S_t+[S,A]=2\frac{a'}{a}S-8a\partial_x^2+\left(32a^3+a''-2\frac{(a')^2}{a}\right)x^2.
\end{displaymath}
If $a$ is positive, even, and a solution of
\begin{displaymath}
32a^3+a''-2\frac{(a')^2}{a}=0\quad\text{in }[-1,1],
\end{displaymath}
then our formal calculations show that
\begin{displaymath}
\partial_t(a^{-1}\partial_t\log H_a(t))\geq0\quad\text{in }[-1,1].
\end{displaymath}
Hence, for $s<t$ we have
\begin{displaymath}
a(t)\partial_t\log H_a(s)\leq a(s)\partial_t\log H_a(t).
\end{displaymath}
Integrating between $[-1,0]$ and $[0,1]$ and using the evenness of $a$, we conclude 
$$H_a(0)\leq H_a(-1)^{1/2}H_a(1)^{1/2}.$$ Now, if $a$ solves
\begin{displaymath}
\left\{\begin{array}{l}
32a^3+a''-2\frac{(a')^2}{a}=0\\ \\a(0)=1, a'(1)=0\end{array}\right.
\end{displaymath}
$a$ is positive, even, and $\lim_{R\to\infty}Ra(R)=0$. Also, $a_R(t)=Ra(Rt)$ also solves the equation. 
If the formal calculation holds for $H_{a_R}$, 
\begin{displaymath}
\left\lVert e^{Rx^2}u(0)\right\rVert^2\leq\left\lVert e^{Ra(R)x^2}u(-1)\right\rVert
\left\lVert e^{Ra(R)x^2}u(1)\right\rVert.
\end{displaymath}
In particular, $u\equiv0$. But $u(x,t)=(t-i)^{-1/2}e^{i|x|^2/4(t-i)}$ is a non-zero free solution, 
which decays as a quadratic exponential at $t=\pm1$.

\section{The case $\alpha\not=\beta$; the conformal or Appel transformation}

\begin{lem}
Assume $u(y,s)$ verifies
\begin{displaymath}
\partial_s u=(a+ib)(\triangle u+V(y,s)u+F(y,s))\quad\text{in }\R^n\times[0,1],
\end{displaymath}
$a+ib\not=0$, $\alpha>0$, $\beta>0$, $\gamma\in\R$ and set
\begin{multline*}
\tilde u(x,t)=\left(\frac{\sqrt{\alpha\beta}}{\alpha(1-t)+\beta t} \right)^{n/2}
u\left(\frac{\sqrt{\alpha\beta}\,x}{\alpha(1-t)+\beta t}, \frac{\beta t}{\alpha(1-t)+\beta t}\right)\\
\times\exp\left(\frac{(\alpha-\beta)|x|^2}{4(a+ib)(\alpha(1-t)+\beta t)}\right).
\end{multline*}
Then $\tilde u$ verifies
\begin{displaymath}
\partial_t \tilde u=(a+ib)(\triangle \tilde u+\tilde V(x,t)\tilde u+\tilde F(x,t))\quad\text{in }\R^n\times[0,1],
\end{displaymath}
\begin{displaymath}
\tilde V(x,t)=\frac{\alpha\beta}{(\alpha(1-t)+\beta t)^2}
V\left(\frac{\sqrt{\alpha\beta}\,x}{\alpha(1-t)+\beta t}, \frac{\beta t}{\alpha(1-t)+\beta t}\right),
\end{displaymath}
\begin{displaymath}
\tilde F(x,t)=\frac{\sqrt{\alpha\beta}}{(\alpha(1-t)+\beta t)^{\frac{n}{2}+2}}
F\left(\frac{\sqrt{\alpha\beta}\,x}{\alpha(1-t)+\beta t}, \frac{\beta t}{\alpha(1-t)+\beta t}\right).
\end{displaymath}
Moreover, if $s=\beta t/(\alpha (1-t)+\beta t)$,
\begin{displaymath}
\left\lVert e^{\gamma|x|^2}\tilde u(t)\right\rVert=
\left\lVert e^{\left[\frac{\gamma\alpha\beta}{(\alpha s+\beta(1-s))^2}+\frac{(\alpha-\beta)a}{4(a^2+b^2)(\alpha s+\beta(1-s))}\right]|y|^2}
u(s)\right\rVert
\end{displaymath}
\begin{displaymath}
\left\lVert e^{\gamma|x|^2}\tilde F(t)\right\rVert=\frac{\alpha\beta}{(\alpha(1-t)+\beta t)^2}
\left\lVert e^{\left[\frac{\gamma\alpha\beta}{(\alpha s+\beta(1-s))^2}+\frac{(\alpha-\beta)a}{4(a^2+b^2)(\alpha s+\beta(1-s))}\right]|y|^2}
F(s).\right\rVert
\end{displaymath}
\end{lem}
 
The proof is by change of variables.

\paragraph{\textbf{Conclusion of the proof of Theorem \ref{thm1}:}}
We can assume $\alpha\not=\beta$. We can also assume $\alpha<\beta$ (change $u$ for $\overline u(1-t)$). (This gives $(\alpha-\beta)a<0$.) As before, $$H=(\triangle+V),\quad u_a=e^{(a+i)tH}u(0)=e^{atH}u(t),\quad a>0.$$ By the ``energy estimate'' we now have
\begin{displaymath}
\left\lVert e^{|x|^2/\alpha^2}u_a(1)\right\rVert\leq e^{a||V||_\infty}\left\lVert e^{|x|^2/\alpha^2}u(1)\right\rVert
\end{displaymath}
and
\begin{displaymath}
\left\lVert e^{|x|^2/\beta^2}u_a(0)\right\rVert\leq \left\lVert e^{|x|^2/\beta^2}u(0)\right\rVert.
\end{displaymath}
We now have also $$\partial_t u_a=(a+i)(\triangle u_a+Vu_a),$$ so when we do the Appel transform, we have, with $\gamma_a=1/\alpha_a\beta_a$,
$$\partial_t \tilde u_a=(a+i)((\triangle +\tilde V^a)\tilde u_a),$$ where
\begin{displaymath}
\tilde V^a(x,t)=\frac{\alpha_a\beta_a}{(\alpha_a (1-t)+\beta_a t)^2}V\left(\frac{\sqrt{\alpha_a\beta_a}\,x}{(\alpha_a (1-t)+\beta_a t}\right).
\end{displaymath}

Now, fo r $a>0$ we have ``log convexity'' in this last problem. Moreover, by the Appel Lemma and our definitions, we have
\begin{displaymath}
\left\lVert e^{\gamma_a|x|^2}\tilde u_a(0)\right\rVert\leq\left\lVert e^{|x|^2/\beta^2} u(0)\right\rVert,\quad
\left\lVert e^{\gamma_a|x|^2}\tilde u_a(1)\right\rVert\leq e^{a||V||_\infty}\left\lVert e^{|x|^2/\alpha^2} u(1)\right\rVert.
\end{displaymath}
Thus, 
\begin{displaymath}
\left\lVert e^{\gamma_a|x|^2}\tilde u_a(t)\right\rVert\leq e^{N(1+M_1+M_1^2)}e^{a||V||_\infty}
\left\lVert e^{|x|^2/\beta^2} u(0)\right\rVert^{1-t}\left\lVert e^{|x|^2/\alpha^2} u(1)\right\rVert^t
\end{displaymath}
and the corresponding ``local smoothing'' estimate. But now, letting $a\to0$ and changing variables our result follows.

\paragraph{\textbf{Time dependent, complex potentials:}} 
We will consider complex potentials $V(x,t)$, $||V||_\infty\leq M_0$. We will also assume 
$$\lim_{R\to0}||V||_{L^1([0,1],L^\infty(|x|>R))}=0.$$ We first recall a result in \cite{KPV3}.

\begin{lem}\label{lemA}
There exists $N=N(n)$, $\epsilon_0=\epsilon_0(n)>0$ so that, if $\vec\lambda\in\R^n$, $V\in L^1([0,1];L^\infty)$, 
$||V||_{L^1([0,1];L^\infty)}\leq\epsilon_0$,
then if $u\in C([0,1];L^2)$ satisfies
\begin{displaymath}
\partial_tu=i(\triangle u+V(x,t)u+F(x,t))\quad\text{in }\R^n\times[0,1],
\end{displaymath}
then
\begin{displaymath}
\sup_{t\in[0,1]}\left\lVert e^{\vec\lambda x}u(t)\right\rVert\leq N\left[\left\lVert e^{\vec\lambda x}u(0)\right\rVert
+\left\lVert e^{\vec\lambda x}u(1)\right\rVert+\left\lVert e^{\vec\lambda x}F\right\rVert_{L_t^1L_x^2}\right].
\end{displaymath}
\end{lem}

\begin{thm}\label{thm2}
Let $V\in L^1_tL_x^\infty$, $\lim_{R\to0}||V||_{L^1([0,1],L^\infty(|x|>R))}=0$. Let $u\in C([0,1];L^2)$ solve
\begin{displaymath}
\partial_tu=i(\triangle u+V(x,t)u)\quad \text{in } \R^n\times[0,1].
\end{displaymath}
 Assume in addition that $V\in L^\infty(\R^{n+1})$, and that 
$$\left\lVert e^{|x|^2/\beta^2} u(0)\right\rVert<\infty,\quad \left\lVert e^{|x|^2/\alpha^2} u(1)\right\rVert<\infty.$$ Then,
$\exists N=N(\alpha,\beta)$ s.t.
\begin{multline*}
\sup_{[0,1]}\left\lVert e^{|x|^2/(\alpha t+(1-t)\beta)^2} u(t)\right\rVert+
\left\lVert\sqrt{t(1-t)} e^{|x|^2/(\alpha t+(1-t)\beta)^2} \nabla u(t)\right\rVert_{L^2(\R^n\times[0,1])}\leq\\
\leq Ne^{N||V||_\infty}\left[
\left\lVert e^{|x|^2/\beta^2} u(0)\right\rVert+\left\lVert e^{|x|^2/\alpha^2} u(1)\right\rVert+
\sup_{[0,1]}||u(t)||\right].
\end{multline*}
\end{thm}

\begin{proof}
We start out by using the Appel transform, $\tilde u(x,t)$ and setting $\gamma=1/\alpha\beta$, $(a+ib)=i$. We now have 
$\tilde u\in C([0,1];L^2)$, $$\partial_t \tilde u=i(\triangle\tilde u+\tilde V(x,t)\tilde u),$$ and it is easy to check that the potential 
$\tilde V$ verifies $$||\tilde V||_\infty\leq \max\left\{\frac{\alpha}{\beta},\frac{\beta}{\alpha}\right\}||V||_\infty$$ and
$\lim_{R\to0}||\tilde V||_{L^1([0,1],L^\infty(\R^n\setminus B_R)}=0$. Also, we have 
\begin{displaymath}
||\tilde u(t)||=||u(s)||,\quad 
||e^{\gamma|x|^2}\tilde u(t)||=||e^{|x|^2/(\alpha s+(1-s)\beta)^2}u(s)||,\!\!\quad s=\frac{\beta t}{\alpha(1-t)+\beta t}. 
\end{displaymath}
Choose now $R>0$ such that $||\tilde V||_{L^1([0,1],L^\infty(\R^n\setminus B_R))}\leq\epsilon_0,$ $\epsilon_0$ as in Lemma \ref{lemA}. Then, 
\begin{displaymath}
\partial_t\tilde u=i(\triangle\tilde u+\tilde V_R(x,t)\tilde u+\tilde F_R(x,t))
\end{displaymath}
where $\tilde V_R(x,t)=\chi_{\R^n\setminus B_R}\tilde V(x,t)$, $\tilde F_R=\chi_{B_R}\tilde V\tilde u$. By the Lemma we have:
\begin{displaymath}
\sup_{t\in[0,1]}\left\lVert e^{\vec\lambda x}\tilde u(t)\right\rVert\leq N\left[\left\lVert e^{\vec\lambda x}\tilde u(0)\right\rVert
+\left\lVert e^{\vec\lambda x}\tilde u(1)\right\rVert+e^{|\vec\lambda|R}\sup_{[0,1]}||\tilde V(t)||
\sup_{[0,1]}||\tilde u(t)||\right].
\end{displaymath}
Now, replace $\vec\lambda$ by $2\vec\lambda\sqrt{\gamma}$, square both sides, multiply be $e^{-|\vec\lambda|^2/2}$ and integrate both sides with respect to $\vec\lambda$ in $\R^n$. Using this and the identity
\begin{displaymath}
\int e^{2\sqrt{\gamma}\vec\lambda x}e^{-|\vec\lambda|^2/2}d\vec\lambda=(2\pi)^{n/2}e^{2\gamma|x|^2},
\end{displaymath}
we obtain the inequality
\begin{multline*}
\sup_{t\in[0,1]}\left\lVert e^{\gamma|x|^2}\tilde u(t)\right\rVert\leq\\
\leq N\left[
\left\lVert e^{|x|^2/\beta^2} u(0)\right\rVert+\left\lVert e^{|x|^2/\alpha^2} u(1)\right\rVert+
\sup_{[0,1]}\left\lVert V(t)\right\rVert
\sup_{[0,1]}\left\lVert u(t)\right\rVert\right]
\end{multline*}
To prove the regularity of $u$, we proceed as follows: the standard Duhamel formula gives
\begin{displaymath}
\tilde u(t)=e^{it\triangle}\tilde u(0)+i\int_0^t e^{i(t-s)\triangle}\left(\tilde V(s)\tilde u(s)\right)ds.
\end{displaymath}
For $0<a<1$, set $$\tilde F_a(t)=\frac{i}{a+i}e^{at\triangle}\left(\tilde V(t)\tilde u(t)\right),$$
\begin{displaymath}
\tilde u_a(t)=e^{(a+i)t\triangle}\tilde u(0)+(i+a)\int_0^t e^{(a+i)(t-s)\triangle}\tilde F_a(t)ds.
\end{displaymath}
Clearly, $$\tilde u_a(t)=e^{at\triangle}\tilde u(t).$$ We now have, from the ``energy estimates'', with 
$\gamma_a=\frac{\gamma}{(1+4\gamma a)}$,
\begin{displaymath}
\sup_{[0,1]}\left\lVert e^{\gamma_a|x|^2}\tilde u_a(t)\right\rVert\leq
\sup_{[0,1]}\left\lVert e^{\gamma|x|^2}\tilde u(t)\right\rVert
\end{displaymath}
\begin{displaymath}
\sup_{[0,1]}\left\lVert e^{\gamma_a|x|^2}\tilde F_a(t)\right\rVert\leq e^{||\tilde V||_\infty}
\sup_{[0,1]}\left\lVert e^{\gamma|x|^2}\tilde u(t)\right\rVert.
\end{displaymath}
But then, our formal ``smoothing effect'' argument applies and gives: (using the first step)(key Lemma)
\begin{multline*}
\left\lVert\sqrt{t(1-t)}\nabla\tilde u_a e^{\gamma_a|x|^2}\right\rVert_{L^2(\R^n\times[0,1])}\leq\\
Ne^{N||V||_\infty}\left[
\left\lVert e^{|x|^2/\beta^2} u(0)\right\rVert+\left\lVert e^{|x|^2/\alpha^2} u(1)\right\rVert+
\sup_{[0,1]}\left\lVert V(t)\right\rVert
\sup_{[0,1]}\left\lVert u(t)\right\rVert\right].
\end{multline*}
We now let $a\to 0$.
\end{proof}

\section{The Hardy uncertainty principle}
 
Recall that for free evolution, $\partial_tu=i\triangle u$, Hardy's uncertainty principle says that if $u(0)\in L^2(e^{2|x|^2/\beta^2}dx)$,
$u(1)\in L^2(e^{2|x|^2/\alpha^2}dx)$, and $\alpha\beta\leq4$, then $u\equiv 0$, and $4$ is sharp. We will now show a (weakened) version of this for all our potentials.

\begin{thm}\label{thm3}
Let $V=V(x)$, $V$ real, $||V||_\infty<\infty$, or $V=V(x,t)$, $V$ complex, $||V||_\infty<\infty$, 
$\lim_{R\to0}||V||_{L^1([0,1],L^\infty(|x|>R))}=0$. Assume that $u\in C([0,1];L^2)$ is a solution of
$$\partial_tu=i(\triangle u+V(x,t)u)\quad \text{in } \R^n\times[0,1],$$ such that $e^{|x|^2/\beta^2}u(0)\in L^2$, $e^{|x|^2/\alpha^2}u(1)\in L^2$,
and $\alpha\beta<2$. Then $u\equiv0$.
\end{thm}

\paragraph{\textbf{Preliminaries:}}
Let $\gamma=1/\alpha\beta$. Using the Appel transform and our convexity and ``smoothing'' estimates we can assume, without loss of generality, that the following holds for $\gamma>1/2$:
\begin{equation}\label{plus}
\sup_{[0,1]}\left\lVert e^{\gamma|x|^2}u(t)\right\rVert_{L^2}+
\sup_{[0,1]}\left\lVert \sqrt{t(1-t)}e^{\gamma|x|^2}\nabla u(t)\right\rVert_{L^2(\R^n\times[0,1])}<\infty.
\end{equation}

Let me first give a formal argument, in the spirit of our ``log convexity'' inequalities. If $e_1=(1,0,\ldots,0)$, $R>0$, set 
$f=e^{\mu|x+Re_1t(1-t)|^2}u$, where $0<\mu<\gamma$, and $H(t)=(f,f)$. At the formal level, it is easy to show (for the free evolution) that
\begin{displaymath}
\partial_t^2\log H(t)\geq-R^2/4\mu,
\end{displaymath} 
so that $H(t)e^{-R^2t(1-t)/8\mu}$ is log convex in $[0,1]$ and so 
\begin{displaymath}H(1/2)\leq H(0)^{1/2}H(1)^{1/2}e^{R^2/32\mu}.\end{displaymath}
Letting $\mu\uparrow\gamma$ we see that
\begin{displaymath}
\int e^{2\gamma\left|x+\frac{Re_1}{4}\right|^2}|u(1/2)|^2\leq\left\lVert e^{\gamma|x|^2}u(0)\right\rVert
\left\lVert e^{\gamma|x|^2}u(1)\right\rVert e^{R^2/32\gamma}.
\end{displaymath}
Thus, 
\begin{displaymath}
\int_{B(\epsilon R/4)} |u(1/2)|^2\leq\left\lVert e^{\gamma|x|^2}u(0)\right\rVert
\left\lVert e^{\gamma|x|^2}u(1)\right\rVert e^{[R^2(1-4\gamma^2(1-\epsilon)^2)]/32\gamma},
\end{displaymath}
$0<\epsilon<1$, which implies $u(1/2)\equiv0$ as $R\to\infty$, ($\gamma>1/2$).

The path from the formal argument to the rigorous one is not easy. We will do it instead with the Carleman inequality:

\begin{lem}
Let $\phi(t)$, $\psi(t)$ be smooth functions on $[0,1]$, $g(x,t)\in C_0^\infty(\R^n\times[0,1])$, $e_1=(1,0,\ldots,0)$. Then, for $\mu>0$, we have (for $R\gg0$),
\begin{multline*}
\int\!\!\!\int[\psi''(t)-\frac{R^4}{32\mu}[\phi''(t)]^2]e^{2\psi(t)}e^{2\mu\left|\frac{x}{R}-\phi(t)e_1\right|^2}|g|^2\leq\\\leq
\int\!\!\!\int e^{2\psi(t)}e^{2\mu\left|\frac{x}{R}-\phi(t)e_1\right|^2}|(i\partial_t+\triangle)g|^2.
\end{multline*}
\end{lem}
\begin{proof}
Let $f=e^{\mu\left|\frac{x}{R}+\phi(t)e_1\right|^2+\psi(t)}g$. Then
\begin{displaymath}
e^{\mu\left|\frac{x}{R}+\phi(t)e_1\right|^2+\psi(t)}(i\partial_t+\triangle)g=S_\mu f+A_\mu f,
\end{displaymath}
where $S_\mu=S_\mu^\ast$, $A_\mu=-A_\mu^\ast$ (the adjoints are now with respect to the $L^2(dxdt)$ inner product), and
\begin{displaymath}
S_\mu=i\partial_t+\triangle+\frac{4\mu^2}{R^2}\left|\frac{x}{R}+\phi e_1\right|^2,
\end{displaymath}
\begin{displaymath}
A_\mu=-\frac{4\mu}{R}\left(\frac{x}{R}+\phi e_1\right)\cdot\nabla-\frac{2\mu n}{R^2}-2i\mu\phi'\left(\frac{x_1}{R}+\phi e_1\right)-i\psi'.
\end{displaymath}
We then have:
\begin{multline*}
\int\!\!\!\int e^{2\psi(t)}e^{2\mu\left|\frac{x}{R}-\phi(t)e_1\right|^2}|(i\partial_t+\triangle)g|^2=\\
=\left<(S_\mu+A_\mu)f,(S_\mu+A_\mu)f\right>
=\left<S_\mu f,S_\mu f\right>+\left<A_\mu f,A_\mu f\right>+\\+\left<S_\mu f,A_\mu f\right>+\left< A_\mu f, S_\mu f\right>
\geq\left<[S_\mu,A_\mu]f,f\right>.
\end{multline*}
We now compute $[S_\mu,A_\mu]$ and obtain:
\begin{multline*}
[S_\mu,A_\mu]=-\frac{8\mu}{R^2}\triangle+\frac{32\mu^3}{R^4}\left|\frac{x}{R}+\phi e_1\right|^2+\\+
2\mu\left(\frac{x_1}{R}+\phi e_1\right)\phi''+2\mu(\phi')^2-\frac{8i\mu\phi'}{R}\partial_{x_1}+\psi''.
\end{multline*}
Thus,
\begin{multline*}
\left<[S_\mu,A_\mu]f,f\right>=\frac{8\mu}{R^2}\int|\nabla_{x'}f|^2+\frac{32\mu^3}{R^4}\int\left|\frac{x}{R}+\phi e_1\right|^2|f|^2+\\
+\frac{8\mu}{R^2}\int|\partial_{x_1}f|^2+2\mu\int\left(\frac{x_1}{R}+\phi e_1\right)\phi''|f|^2+\\
+2\mu\int(\phi')^2|f|^2-
\frac{8i\mu}{R}\int\phi'\partial_{x_1}f\overline f+\int\psi''|f|^2=\\
=\frac{8\mu}{R^2}\int|\nabla_{x'}f|^2+8\mu\int\left|\frac{i}{R}\partial_{x_1}f-\frac{\phi'}{2}f\right|^2+\\
+\frac{32\mu^3}{R^4}\int\left|\frac{x}{R}+\phi e_1\right|^2|f|^2
+2\mu\int\left(\frac{x_1}{R}+\phi e_1\right)\phi''|f|^2+\int\psi''|f|^2=\\
=\frac{8\mu}{R^2}\int|\nabla_{x'}f|^2+8\mu\int\left|\frac{i}{R}\partial_{x_1}f-\frac{\phi'}{2}f\right|^2+\\
+\frac{32\mu^3}{R^4}\int\left|\frac{x}{R}+\left(\phi+\frac{R^4}{32\mu^2}\phi''\right) e_1\right|^2|f|^2-
\frac{R^4(\phi'')^2}{32\mu}\int|f|^2+\psi''\int|f|^2,
\end{multline*}
and the Lemma follows.
\end{proof}

Next, choose $\phi(t)=t(1-t)$, $\psi(t)=-(1+\epsilon)\frac{R^4}{16\mu}t(1-t)$. Then
\begin{displaymath}
\psi''(t)-\frac{R^4}{32\mu}(\phi'')^2(t)=\frac{(1+\epsilon)}{8\mu}R^4-\frac{R^4}{8\mu}=\frac{\epsilon}{8\mu}R^4
\end{displaymath}
and so our inequality reads, for $g\in C_0^\infty(\R^n\times[0,1])$, 
\begin{multline*}
\frac{\epsilon}{8\mu}R^4\int\!\!\!\int e^{2\psi(t)}e^{2\mu\left|\frac{x}{R}+\phi(t)e_1\right|^2}|g|^2\leq
\int\!\!\!\int e^{2\psi(t)}e^{2\mu\left|\frac{x}{R}+\phi(t)e_1\right|^2}|(i\partial_t+\triangle)g|^2.
\end{multline*}
We next fix $R>0$, recall that $u$ solves $i\partial_tu+\triangle u=Vu$, and that the estimates \eqref{plus} hold. Choose then
$\eta(t)$, $0\leq\eta\leq1$, $\eta\equiv1$ where $t(1-t)\geq 1/R$, $\eta\equiv0$ near $t=1,0$, so that 
$$\supp\eta'\subset\{t(1-t)\leq 1/R\},\quad |\eta'|\leq C R.$$ Choose also $M\gg R$, $\theta\in C_0^\infty(\R^n)$, and now set
$g(x,t)=\eta(t)\theta(x/M)u(x,t)$, which is compactly supported in $\R^n\times(0,1)$, so that our estimate holds.
\begin{displaymath}
\begin{array}{ccccccc}
(i\partial_t+\triangle)g&=&Vg&+&i\eta'(t)\theta\left(\frac{x}{M}\right)u&+&
\left(\frac{1}{M^2}\triangle\theta\left(\frac{x}{M}\right)u+\frac{2\nabla\theta\left(x/M\right)\cdot\nabla u}{M}\right)\\
 &=&I&+&II&+&III.
\end{array}
\end{displaymath}
Finally, let $\mu=(1+\epsilon)^{-3}\gamma R^2$. Our inequality then gives:
\begin{multline*}
\frac{\epsilon}{8}\frac{(1+\epsilon)^3}{\gamma}R^2\int\!\!\!\int e^{2\psi(t)}e^{2\mu\left|\frac{x}{R}+\phi(t)e_1\right|^2}|g|^2\leq\\\leq
\int\!\!\!\int e^{2\psi(t)}e^{2\mu\left|\frac{x}{R}+\phi(t)e_1\right|^2}\left\{I+II+III\right\}.
\end{multline*}
The contribution of $I$ to the right hand side is bounded by 
\begin{displaymath}
||V||_\infty\int\!\!\!\int e^{2\psi(t)}e^{2\mu\left|\frac{x}{R}+\phi(t)e_1\right|^2}|g|^2,
\end{displaymath}
so that, if $R$ is very large, we can hide it in the left side, to see that we only have to deal with $II$ and $III$. Recall that 
$\psi(t)=(1+\epsilon)\frac{R^4}{16\mu}t(1-t)\leq0$, so $e^{2\psi(t)}\leq1$. On the support of $\eta'$, we have $t(1-t)\leq1/R$, so that
$0\leq\phi(t)\leq1/R$. We now estimate
\begin{multline*}
2\mu\left|\frac{x}{R}+\phi(t)e_1\right|^2=\frac{2\gamma R^2}{(1+\epsilon)^3}\left\{\left|\frac{x}{R}\right|^2+2\frac{x_1}{R}\phi(t)+
\phi(t)^2\right\}=\frac{2\gamma}{(1+\epsilon)^3}|x|^2+\\+\frac{2\gamma}{(1+\epsilon)^3}x_1R\phi(t)+
\frac{2\gamma}{(1+\epsilon)^3}R^2\phi(t)^2\leq 2\gamma|x|^2+C_\epsilon,
\end{multline*}
on $\supp\eta'$, where $\phi(t)\leq1/R$. Thus, because of \eqref{plus}, the contribution of $II$ is bounded by $C_\epsilon R$. The contribution of $III$ is controlled by (recalling that $\eta\equiv0$ when $t(1-t)\leq\frac{1}{2}\frac{1}{R}$)
\begin{multline*}
\frac{C}{M^4}\int\!\!\!\int_{|x|\leq 2M} |u(x,t)|^2e^{2\psi(t)}e^{2\mu\left|\frac{x}{R}+\phi(t)e_1\right|^2}+\\+
\frac{C}{M^2}\int\!\!\!\int_{|x|\leq 2M} |\nabla u(x,t)|^2e^{2\psi(t)}e^{2\mu\left|\frac{x}{R}+\phi(t)e_1\right|^2}t(1-t)R.
\end{multline*}
If we use \eqref{plus}, $\psi\leq0$, the bound above for $2\mu\left|x/R+\phi(t)e_1\right|^2$ becomes
\begin{displaymath}
\leq \frac{2\gamma|x|^2}{(1+\epsilon)^3}+\frac{2\gamma|x_1|}{(1+\epsilon)^3}\frac{R}{4}+\frac{2\gamma}{(1+\epsilon)^3}\frac{R^2}{16}
\leq2\gamma|x|^2+ C_{\epsilon,R}.
\end{displaymath}
Thus, letting $M\to\infty$, we see that, for fixed $R$, $III\to0$, so that, since $\eta\equiv1$ on $t(1-t)\geq 1/R$, we obtain:
\begin{displaymath}
\frac{\epsilon}{8}\frac{(1+\epsilon)^3}{\gamma}R^2\int\!\!\!\int_{t(1-t)\geq1/R} 
e^{2\psi(t)}e^{2\mu\left|\frac{x}{R}+\phi(t)e_1\right|^2}|u|^2\leq C_{\gamma,\epsilon}R.
\end{displaymath}
We are now going to restrict to integration over the region $\left|\frac{x}{R}\right|\leq\delta$, $\left|t-1/2\right|\leq\delta$, where 
$\delta$ is small, to be chosen. Then,
\begin{displaymath}
\left|\frac{x}{R}+\phi(t)e_1\right|\geq\left|\phi\left(\frac{1}{2}\right)\right|-2\delta=\left(\frac{1}{4}-2\delta\right),
\end{displaymath}
so that 
\begin{displaymath}
\left|\frac{x}{R}+\phi(t)e_1\right|^2\geq\frac{1}{16}-2\delta\left(\frac{1}{2}-2\delta\right).
\end{displaymath}
\begin{multline*}
\psi(t)=\psi\left(\frac{1}{2}\right)+\left[\psi(t)-\psi\left(\frac{1}{2}\right)\right]\geq\psi\left(\frac{1}{2}\right)
-|\psi'(\theta)|\delta\geq\\\geq-(1+\epsilon)\frac{R^4}{16\mu}\frac{1}{4}-\delta(1+\epsilon)\frac{R^4}{16\mu},
\end{multline*}
so that, in our region of integration,
\begin{multline*}
2\mu\left|\frac{x}{R}+\phi(t)e_1\right|^2+2\psi(t)\geq\\\geq\frac{2\mu}{16}-\frac{2(1+\epsilon)R^4}{16\mu4}-
4\delta\mu\left(\frac{1}{2}-2\delta\right)-\delta(1+\epsilon)\frac{R^4}{16\mu}=
\frac{2}{16}\frac{\gamma R^2}{(1+\epsilon)^3}-\\-\frac{2}{16}\frac{(1+\epsilon)^4}{\gamma 4}R^2-C\delta R^2=
\frac{2}{16}R^2\left[\frac{\gamma}{(1+\epsilon)^3} - \frac{(1+\epsilon)^4}{4\gamma}-C\delta\right],
\end{multline*}
since $\mu=\frac{\gamma}{(1+\epsilon)^3}R^2$.
But, if $\gamma>1/2$, $$\frac{\gamma}{(1+\epsilon)^3} - \frac{(1+\epsilon)^4}{4\gamma}>0,$$ for some $\epsilon$ small, and so, for $\delta$ smaller than that we get a lower bound of $C_{\epsilon,\delta}R^2$. We thus have
\begin{displaymath}
C_{\epsilon,\delta}R\int_{|t-1/2|\leq\delta}\int_{\left|\frac{x}{R}\right|\leq\delta}|u|^2 e^{C_{\epsilon,\delta}R^2}\leq 
C_{\epsilon,\delta}.
\end{displaymath}
But then, since
\begin{multline*}
\int_{|t-1/2|\leq\delta}\int_{\left|\frac{x}{R}\right|>\delta}|u|^2=
\int_{|t-1/2|\leq\delta}\int_{\left|\frac{x}{R}\right|\leq\delta}e^{2\gamma|x|^2}e^{-2\gamma|x|^2}|u|^2\leq\\\leq
e^{-2\gamma\delta^2R^2}\int_{|t-1/2|\leq\delta}\int e^{2\gamma|x|^2}|u|^2\leq
C_\gamma e^{-2\gamma\delta^2R^2}
\end{multline*}
by \eqref{plus}, we see that, for appropriate $C_{\gamma,\epsilon,\delta}$ we have
\begin{displaymath}
\left(\int_{|t-1/2|\leq\delta}\int |u|^2 \right)e^{C_{\gamma,\epsilon,\delta}R^2}\leq C_{\gamma,\epsilon,\delta}.
\end{displaymath}
Letting $R\to\infty$, we see that $u\equiv0$ on $\{(x,t):|t-1/2|\leq\delta\}$, therefore $u\equiv0$.

\end{document}